\documentclass[10pt]{amsart}
\usepackage{amsmath,amsfonts,amsthm,amssymb,amscd, soul, color}
\usepackage{geometry}
\usepackage{mathtools}
\DeclarePairedDelimiter{\ceil}{\lceil}{\rceil}

\def\classification#1{\def\@class{#1}}
\classification{\null}

\DeclareFontFamily{OT1}{rsfs}{}
\DeclareFontShape{OT1}{rsfs}{n}{it}{<-> rsfs10}{}
\DeclareMathAlphabet{\mathscr}{OT1}{rsfs}{n}{it}

\newcommand{\R}{{\mathbb R}}
\newcommand{\Pro}{{\mathbb P}}

\newcommand{\Q}{\mathbb{Q}}

\newcommand{\C}{\mathbb{C}}
\newcommand{\F}{\mathbb{F}}
\newcommand{\E}{\mathbb{E}}

\def\u{\underline}

\def\E{\mathsf {E}}
\def\P{\mathcal {P}}

\def\i{\iota}

\newtheorem{theorem}{Theorem}

\newtheorem{lemma}[theorem]{Lemma}
\newtheorem{corollary}[theorem]{Corollary}
\theoremstyle{remark}
\newtheorem{remark}[theorem]{Remark}

\title{Point-plane incidences and some applications in positive characteristic}

\author{Misha Rudnev}
\address{Misha Rudnev, Department of Mathematics, University of Bristol,
  Bristol BS8 1TW, United Kingdom}
\email{m.rudnev@bristol.ac.uk}


 \newenvironment{dedication}
        {\vspace{6ex}\begin{quotation}\begin{center}\begin{em}}
        {\par\end{em}\end{center}\end{quotation}}
\begin{document}
\begin{dedication}
\hspace{4cm}{To the memory of Galya, who would always}\\
\hspace{3cm}{ask me if I've written a paper lately.}
\vspace*{3cm}
\end{dedication}

\begin{abstract} The point-plane incidence theorem states that the number of incidences between $n$ points and $m\geq n$ planes in the projective three-space over a field $F$,  is
$$O\left(m\sqrt{n}+ m k\right),$$
where $k$ is the maximum number of collinear  points, with the extra condition $n< p^2$ if $F$ has characteristic $p>0$. 
This theorem also underlies a state-of-the-art Szemer\'edi-Trotter type bound for point-line incidences in $F^2$, due to  Stevens and de Zeeuw. 

This review focuses on some recent, as well as new, applications of these bounds that lead to progress in several open geometric questions in $F^d$, for $d=2,3,4$. These are the problem of the minimum number of distinct nonzero values of a non-degenerate bilinear form on a point set in $d=2$, the analogue of the Erd\H os distinct distance problem in $d=2,3$ and additive energy estimates for sets, supported on a paraboloid and sphere in $d=3,4$. It avoids discussing sum-product type problems (corresponding to the special case of incidences with Cartesian products), which have lately received more attention.
\end{abstract}

\maketitle

\section{Introduction} 

This paper is centred around the author's point-plane incidence theorem -- the forthcoming Theorem \ref{mish}  \cite[Theorem 3]{Rud} -- in $\Pro^3$, the projective three-space over a field $F$. The notation  $\Pro^3$ will usually appear on its own; over which field it is meant should be clear from the context. In the case when $F=\R$, the reals, as well as the complex field $\C$, somewhat stronger theorems than Theorem \ref{mish} are known, see e.g. \cite{BK}, \cite{ET}.  Hence, one may implicitly assume that $F$ has a large, and therefore odd positive characteristic $p$, which serves as an asymptotic parameter. Since the applicability of the theorem is constrained in terms of $p$, it will often be the case that $F=\F_p$, the prime residue field. 

The standard asymptotic symbols $\gg,\ll,\sim$ are used throughout to subsume absolute constants in inequalities or approximate equalities, as well as, respectively, the symbols $\Omega,O, \Theta$. The symbols $\gtrsim,\lesssim$ also suppress functions growing slower than any power of an asymptotic parameter in inequalities -- which parameter it is should clear from the context.

\medskip
The point-plane theorem is the following statement.

\begin{theorem}  \label{mish} Let $Q, \Pi$ be, respectively, finite sets of points and planes in  $\Pro^3$, with cardinalities $|Q|\leq |\Pi|$, and
$I(Q,\Pi):=\{(q,\pi)\in Q\times \Pi:\,q\in\pi\}$ -- the set of their incidences. If $F$ has positive characteristic $p$, assume $|Q|< p^2$. Let $k$ be the maximum number of collinear points in $Q$. Then
\begin{equation}\label{pups}
|I(Q,\Pi)| \ll |\Pi|(\sqrt{|Q|}+ k).\end{equation}
\end{theorem}
The statement of the theorem can be reversed in an obvious way, using duality in the case when the number of points exceeds the number of planes. Moreover, owing to linearity of the main estimate in $|\Pi|$, $\Pi$ can be a multiset, as long as its cardinality as a set is $\Omega(|Q|)$.

\subsection*{Pedigree}
Results discussed below can be viewed as part of the recent landscape change that has affected the status of many questions in arithmetic and geometric combinatorics. What is behind it has been commonly referred to as the {\em Polynomial method}, with its breakthrough development by Guth and Katz, in particular in their remarkable paper \cite{GK}, which resolved the long-standing Erd\H os distinct distance conjecture in $\R^2$.

The latter paper developed two important theorems, bounding the number of pair-wise intersections of lines in three dimensions, subject to some natural constraints. The First Guth-Katz theorem \cite[Theorem 2.10]{GK}, is in essence algebraic. It adopts the polynomial method in a way somewhat similar to the groundbreaking work of Dvir \cite{D}, and then proceeds by taking advantage of basic properties of ruled surfaces in $\mathbb C^3$. The theory of ruled surfaces can be viewed as one of the foundations of what today may be referred to as  the ``XIX century algebraic geometry'', which Guth and Katz succeeded in rediscovering and relating to discrete geometry questions of today. Other results, such as Theorem \ref{mish}, as well as recent novel developments in incidence theory by, e.g., Sharir and Solomon \cite{SS} have also benefitted by such rediscovery.


The proof of the Second Guth-Katz theorem \cite[Theorem 2.11]{GK} offered the method of polynomial partitioning of the real space, based on the Borsuk-Ulam theorem. Partitioning has been a strategy of choice to approach many real discrete geometry questions, going back at least as far as the vintage proofs of the Szemer\'edi-Trotter theorem \cite{ST}, \cite{CEGSW}. Polynomial partitioning enhances it with unprecedented robustness and flexibility, having generated a massive body of applications and progress towards many open discrete geometry questions in the real space -- see, e.g., \cite{Z}. One testimony to the powers of the technique is that it enables an induction proof of a slightly weaker version of the First Guth-Katz theorem over the reals \cite{Guth}. Nonetheless, being specific for reals, polynomial partitioning is not discussed here any further.

It is the First Guth-Katz theorem that is a key {\em Ursprung} of the results in this review. Even though the original \cite[Proof of Theorem 2.10]{GK} took place in $\mathbb R^3$, it became agreed in  the folklore that the proof should work, with some constraints, over a general field. The first ``official'' account of this was given  by Ellenberg and Hablicsek \cite{EH} in late 2013, followed by Koll\'ar \cite{Ko} and the author \cite{Rud} in 2014. The latter two had been aware of a 2003 paper by Voloch \cite{Vo}, which discussed the constraints under which the key element of the First Guth-Katz theorem proof, the Monge-Salmon theorem \cite{Sa}, applied in positive characteristic.


\begin{theorem}[First Guth-Katz theorem] \label{gkt} Let $L$ be a set of lines in $\C^3$. Suppose,  no more then two lines are concurrent. Then the number of pair-wise intersections of lines in $L$ is 
$$
O\left(|L|^{\frac{3}{2}}+ |L|k\right),
$$
where $k$ is the maximum number of lines, contained in a plane or  ruled quadric.
\end{theorem}

The first major step of the Guth-Katz proof of the Erd\H os distinct distance was due to Elekes and Sharir \cite{ES}. Following Elekes' {\em Budapester Program} \cite{El}, they interpret the number of pairs of congruent segments with endpoints in a plane point set as the number of pair-wise intersections of lines in $\R^3$. Indeed, two segments have the same length if and only if one can be moved to another by a rigid motion from the Special Euclidean Group $SE_2$, and the set of all group elements, taking one endpoint to the other, is geometrically a line in the three-dimensional space $SE_2\subset \Pro^3$. 

The polynomial method would then trap a large number of lines to lie in a fairly low degree algebraic surface. Given a (complex) algebraic surface of degree $\geq 2$, the fact that there are two lines, contained in the surface and intersecting at some point on the surface does not tell one much about this point (the two lines would coincide with asymptotic lines at this point); but points on the surface, where three of more lines meet must lie on a lower-dimensional subvariety. This is why Guth and Katz had to consider Theorem \ref{gkt} as a separate scenario of their general line-line incidence theorem in $\R^3$.

However, if one just thinks of Theorem \ref{gkt} as an incidence theorem, where can the set of lines  $L$, satisfying its apparently stringent no three-concurrency assumption come from? A heuristic (and retrospective) answer would be -- when $L$ can be mapped to some three-dimensional subvariety of the four-dimensional space of lines in $\Pro^3$, also known as the Pl\"ucker-Klein (or just Klein) quadric $\mathfrak K\subset \Pro^5$. The rich theory of the Pl\"ucker-Klein quadric originated in \cite{Pl}, for modern exposition see, e.g. \cite{PW}, \cite{JS}. 

The space $\Pro^3$ and its dual are certainly three-dimensional, so all it takes is to map them into $\mathfrak K$ in the right way, and there is a natural way of doing this, since ``physical'' points and planes in $\Pro^3$ correspond to two canonical rulings of $\mathfrak K$ by two-planes.  In fact, Theorem \ref{mish} came about from studying the three-dimensional variety of lines in the group $SL_2$, which geometrically is a transverse intersection of $\mathfrak K$ by a hyperplane in $\Pro^5$.  Hence, Theorem \ref{mish} can be recast as a line-line incidence in $SL_2$, see Corollary \ref{sl2lines} below. The question came about in the attempt to produce an erratum to a claim in \cite{IRR}, which applied the Guth-Katz approach to the Erd\H os distance problem in $\R^2$ to a similar-sounding question of what is the minimum number of  areas of triangles, rooted at a fixed origin, the other two vertices lying in a non-collinear set  of $n$ points in $\R^2$. The conjecture that this number is roughly $n$ -- modulo an absolute constant and possibly a power of $\log n$ -- is wide open; the long version of the erratum is \cite{IRRE} claims a much more modest partial result, slightly improving the bound $\Omega(n^{\frac{2}{3}})$ over the reals which follows immediately from the Szemer\'edi-Trotter theorem. On the other hand, Theorem \ref{mish} not only enables one to extend,  in generality,  the bound $\Omega(n^{\frac{2}{3}})$ to the positive characteristic case -- see Section \ref{forms} below -- but with some work (not presented here for its technical challenge)  prove a better exponent than  $\frac{2}{3}$ \cite[Theorem 4]{57}.

\subsection{Outline of the paper} The exposition proceeds with two preliminary sections, in preparation for applications in $F^d$ geometry, $d=3,4$.  The main body of Section \ref{prep} presents several more technical restatements of Theorem \ref{mish} as well as its implications for point-line incidence bounds in the plane, developed by Stevens and de Zeeuw \cite{SdZ}.
As two separate short subsections within Section \ref{prep}  (these can be skipped by a reader more interested in applications) we discuss sharpness of Theorem \ref{mish} and its corollaries and outline the main geometric idea behind the proof of the theorem.

After that Section \ref{null} addresses separately the issue of isotropic lines in $F^d,$ $d=3,4$, arising throughout the applications, except the one in Section \ref{forms}.

Section \ref{null} is followed by two sections of applications. Section \ref{erdos} deals with two outstanding Erd\H os-type questions and Section \ref{en} with energy estimates arising in the Fourier analysis perspective, although avoiding Fourier analysis per se.
\section{Other statements of Theorem \ref{mish} and point-line incidence bound}\label{prep}
There are several applications of the point-plane incidence bound when there is a set $L^*$ of ``forbidden'' lines in $\Pro^3$, incidences supported on which can be interpreted in a specific way, and therefore discounted. The purpose of this will be to lower the value of the parameter $k$ in Theorem \ref{mish}, standing for the maximum number of collinear points (planes).

Formally speaking, suppose, there is  a finite set of lines $L^*$ in $\Pro^3$. Define the restricted set of incidences between a point set $Q$ and set of planes $\Pi$ as
\begin{equation}\label{inss}
I^*(Q,\Pi) = \{(q,\pi)\in Q\times \Pi: q\in \pi \mbox{ and } \forall l\in L^*, \,q\not \in l \mbox{ or } l \not\subset\pi\}.
\end{equation}

\addtocounter{theorem}{-2}

\renewcommand{\thetheorem}{\arabic{theorem}A}
\begin{theorem} \label{mishh} Let $Q, \Pi$ be finite sets of points and planes in  $\Pro^3$, with $|Q|\leq |\Pi|$ and $|Q|< p^2$ if $F$ has positive characteristic $p$. For a finite set of lines $L^*$, let  $k^*$ be the maximum number of  points, incident to any line not in $L^*$. 

Then
\begin{equation}\label{pupss}
|I^*(Q,\Pi)| \ll |\Pi| (\sqrt{|Q|}+ k^*).\end{equation}
\end{theorem}

\addtocounter{theorem}{-1}

\renewcommand{\thetheorem}{\arabic{theorem}B}

For applications over the prime residue field $\F_p$ there is the following asymptotic version. See \cite[Theorem 8]{57} and \cite[Section 3]{MuPet1} for its (easy) derivation from Theorem \ref{mish}.
\begin{theorem}
	\label{Misha+}
	Let  $Q$ be a set of points and $\Pi$ a set of planes in $\F^3_p$. Suppose that $|Q| \le |\Pi|$ and that $k$ is the maximum number of collinear points in $Q$. Then
	$$
	|I(Q,\Pi)|
	- \frac{|Q| |\Pi|}{p} \ll  |\Pi|(\sqrt{|Q|} + k) \,.
	$$
	If $L^*$ is a set of lines in $\F_p^3$ and one excludes incidences $(q,\pi)\in Q\times \Pi$, such that $q\in l\subset \pi$ for some $l$ in $L^*$, then  $k$ can be replaced by the maximum number $k^*$ of points of $Q$, supported on a line not in $L^*$. 
\end{theorem}

\addtocounter{theorem}{-1}
\renewcommand{\thetheorem}{\arabic{theorem}C}

\medskip
As it often happens with incidence theorems, one may need a (less efficient) weighted version established via an easy rearrangement argument. We state one variant to be used in the sequel. To each point $q\in Q$ and each plane $\pi\in \Pi$ one assigns, respectively, positive integer weights $w(q),w(\pi)\leq w_0$, for some maximum weight $w_0$. Suppose, the total weight of both sets $Q$ and $\Pi$ equals $W$.  An incidence $q\in \pi$ contributes $w(q)w(\pi)$ to the total number of weighted incidences, denoted as $I_w$. Then one can take  a subset $Q'$ of $\ceil{W/w_0}$ points in $Q$, maximising, over all subsets of $Q$ of this size, the total weight of all planes in $\Pi$ incident to it, and then reassign to each $q\in Q'$ the maximum weight $w_0$. Let $I_w'$ be the number of weighted incidences of the plane set $\Pi$ with $Q'$ instead of $Q$. Clearly, $I'_w\geq I_w$, as well as $|Q'|\ll |\Pi|$. Hence one has the following claim.
\begin{theorem}\label{wmish} 
 Let $Q, \Pi$ be  weighted sets of points and planes  in  $\Pro^3$, both with total weight $W$. Suppose, maximum weights are bounded by $w_0\geq 1$. Let $k$ be the maximum number of collinear points, counted without weights. Suppose, $\frac{W}{w_0}< p^2$ if $p>0$ is the characteristic of $F$.
Then 
\begin{equation}\label{pupsweight}
I_w\ll W(\sqrt{w_0W}+ k w_0).\end{equation}
The same estimate holds for the quantity $I^*_w$, which discounts weighted incidences along a certain set $L^*$ of lines in $\Pro^3$, with $k$ replaced by $k^*$ -- the maximum number of points in $Q$  incident to a line not in $L^*$.
\end{theorem}
Observe that if there was an a-priori information on the distribution of weight among the points/planes, one could take it into account by dyadic partitioning and applying Theorem \ref{wmish} ``locally'' to dyadic groups, similar,  to applications of the Szemer\'edi-Trotter theorem in, e.g. \cite[Lemma 6]{IKRT}. However, such an opportunity has not come about so far in applications of the point-plane bound.

Theorem \ref{mish} has recently found many applications in sum-product type estimates in, e.g., \cite{RRS},  \cite{AMRS}, \cite{57} where the arising sets of points and planes have natural structure of Cartesian products. In particular, in \cite[Corollary 6]{AMRS}, it was observed that Theorem \ref{mish} implied a point-line incidence bound in $F^2$ in the special case of the point set being a Cartesian product. 

Stevens and de Zeeuw \cite{SdZ} derived a stronger bound in the latter case, as follows.

\addtocounter{theorem}{1}
\renewcommand{\thetheorem}{\arabic{theorem}}
\begin{theorem}[\cite{SdZ}, Theorem 4] \label{SdZcp}
Let $A,B \subset F$ with $|A| \leq |B|$ and let $L$ be a collection of lines in $F^2$, if  $F$ has positive characteristic $p>0$, assume 
$|A||L| < p^2.$

Then the set of incidences
$I(Q,L)$ between the point set $Q=A \times B\subset F^2$ and $L$ satisfies the bound
$$|I(Q, L)| \ll |A|^{3/4}|B|^{\frac{1}{2}}|L|^{3/4} +|Q|+|L|.
$$
\end{theorem}
Once one has Theorem \ref{SdZcp}, it can be used iteratively to yield a general point-line incidence theorem, owing to a structural observation made in the foundational paper by Bourgain, Katz  and Tao \cite[Section 6]{BKT} which was followed up on and cast into a quantitative form by Jones \cite{J}. The observation is that a large part of a putative point set in $F^2$ with too many incidences with a set of lines of roughly the same size should be contained in a Cartesian product-like structure. Although the implementation of this is relatively costly from the quantitative point of view, it is by an order of magnitude stronger than the previously known best point-line incidence bound in $\F_p^2$ by Jones \cite{J}, which was derived from earlier sum-product bounds due to the arithmetic subterfuge of {\em additive pivot} founded in \cite{BKT}. 

\begin{theorem}[\cite{SdZ}, Theorem 3]\label{SdZgen}
	The set of incidences
$I(Q,L)$ between sets $Q,\,L$ of respectively points and lines in  $F^2$ satisfies the bound \begin{equation}\label{ins} |I(Q,L)| \ll 
	(|Q||L|)^{\frac{11}{15}} + |Q|+|L| \qquad for\;\;\;|Q|^{13}|L|^{-2}< p^{15}\,. \end{equation} \end{theorem}
	
In positive characteristic the bounds of Theorems \ref{mish} through \ref{SdZgen} will be referred to as the {\em small set case}, that is they hold under some $<$ inequality constraints  in terms of $p$. The complementary  {\em large set case} has been approached in the finite field case -- in particular in the context of applications discussed further in this review -- via  eigenvalue linear algebra-based techniques, effected by the use of character sums or spectral graph lemmata. See, e.g. \cite{Io}, \cite{IR}, \cite{HI} \cite{mamma}, \cite{IKL}, \cite{Vinh}.  In particular, the latter work by Vinh \cite[Theorem 3]{Vinh} established a finite field point-line incidence bound, which in the $\F_p^2$ context states
 \begin{equation}\label{insv}	|I(Q,L) | \leq \frac{|Q||L|}{p} +  \sqrt{p|Q||L|}\,.\end{equation}

	 \subsection{Sharpness of Theorem \ref{mish}}
There are some examples where the bound of Theorem \ref{mish} is tight. One basic example is as follows. Let $F=\F_p$, take $Q=\mathcal S^2_1$, the unit sphere so $|Q|\sim p^2$. A positive proportion of planes in $F^3$ will meet $Q$ in a conic, which has $\sim p$ points. Hence, the number of incidences is $\Omega(|\Pi|\sqrt{|Q|})$. Moreover, suppose where $p\equiv  3 \pmod 4$, so by the forthcoming Lemma \ref{easy} at most two points are collinear. This examples easily generalises to $Q$ being a two-dimensional bounded degree irreducible variety; if the variety contains lines, one can forbid incidences along these lines and use Theorem \ref{mishh}.

In another example, discussed in detail in \cite[Section 6.2]{Rud}, one considers the set $S$ of points with co-prime coordinates in $[1,\ldots,N]^2$, with  $N< \frac{1}{2}\sqrt{p}$ and the equation $s\cdot t=s'\cdot t$ with variables in $S$. The number of solutions of this equation is bounded from above, by Theorem \ref{mish}, as $O(|S|^3),$ as well as from below as $\Omega(|S|^3),$ which follows by Cauchy-Schwarz inequality, since one knows that all the dot products have values in $[1,\ldots,4N^2]$.

Stevens and de Zeeuw \cite[Example 5]{SdZ} illustrate tightness of Theorem \ref{SdZcp} by matching it with the lower bound in the well-known example by Elekes \cite{Ele}, often used to illustrate tightness of the Szemer\'edi-Trotter theorem. However the Cartesian products $A\times B$ representing point and  line sets in this well-known construction are very uneven, with $|B|\sim|A|^2$ (each line containing $|A|$ points, so the number of incidences is $\sim|A|^4$).

Iteration of Theorem \ref{SdZcp} into Theorem \ref{SdZgen} is quantitatively costly, hence there is hardly a nontrivial instance of tightness of Theorem \ref{SdZgen}.

\subsection{On the proof of Theorem \ref{mish}.}
We do not aim to present a coherent proof here, however will attempt to describe the main idea of how Theorem \ref{mish} gets reduced to a variant of Theorem \ref{gkt}, which, as pointed out earlier, holds over any field $F$, with a constraint $|L|<p^2$ in positive characteristic. It suffices to consider both statements in the algebraic closure of  $F$, or equivalently assume henceforth that $F$ is algebraically closed, in particular infinite.

If a point $q$ lies in a plane $\pi$, there is a pencil of lines incident to $q$ and contained in $\pi$, geometrically a $\Pro^1$. Hence, one moves from the ``physical space'' $\Pro^3$ to the space of lines in $\Pro^3$. The latter is the four-dimensional Klein quadric $\mathfrak K\subset \Pro^5;$ one may think of it as the ``phase space''. The Klein map takes a line $l\subset \Pro^3$ one-to-one to a point $\mathfrak l\in \mathfrak K$.
See \cite[Chapter 6]{JS} or \cite[Chapter 2]{PW} for detail, starting with Pl\"ucker coordinates, that we attempt to avoid in this informal exposition.

The Klein map  takes the set of all physical lines $l$ incident to the ``physical'' point $q$  to a two-plane $\alpha_q\subset \mathfrak K$. Indeed, the set of all lines incident to $q$, viewed projectively is a copy of $\Pro^2$. Similarly, the Klein image of the set of all physical lines $l$ incident to the plane $\pi$  is a two-plane $\beta_q\subset \mathfrak K$. Thus $\mathfrak K$ has two rulings by two-planes, referred to as $\alpha$ and $\beta$-planes, the variety of each ruling being $\Pro^3$. Two planes of the same type always meet at a point in $\mathfrak K$, for there is a unique physical line incident to two distinct physical points. Two planes $\alpha_q$ and $\beta_\pi$ in $\mathfrak K$ meet if and only if $q\in \pi$, this happens along the line in $\mathfrak K$, which is the Klein image of the physical line pencil in $\pi$ via $q$. Thus the space of physical point-plane incidences  in $\Pro^3$ is mapped to the five-dimensional variety of all lines in $\mathfrak K$.

Given two finite sets $\{\alpha_{q\in Q}\}$ and $\{\beta_{\pi\in \Pi}\}$ of two-planes of the two types in $\mathfrak K$, the number of incidences $|I(Q,\Pi)|$ equals the number of lines in pair-wise intersections of the two sets of two-planes:
$$
|I(Q,\Pi)| = | \{ (q,\pi) \in Q\times \Pi: \; \alpha_q\cap\beta_\pi \neq\emptyset\}|.
$$

Next one chooses a generic hyperplane $H\subset \Pro^5$ in the phase space, which will meet $\mathfrak K$, in such a way that (i) the intersection of $H$ with each of the finite two-planes in $\{\alpha_{q\in Q}\}$ and $\{\beta_{\pi\in \Pi}\}$ is a line -- these lines are further referred to $\alpha$ and $\beta$-lines in $\mathfrak K \cap H$, and (ii) $H$ does not contain any of the $\leq |Q|^2+|\Pi|^2$ points of pair-wise intersection of two-planes of the same type. Since $F$ is algebraically closed, the supply of such $H$ is infinite.

This having been done, one now deals with a bi-partite version of Theorem \ref{gkt}, aiming to get a bound on the number of pair-wise intersections of $|Q|+|\Pi|$ lines in a three-quadric $\mathfrak K\cap H$. 
By the choice of $H$, the main condition of Theorem \ref{gkt} that at most two lines meet at a point is satisfied. Unless $H$ is tangent to $\mathfrak K$ at some point $\mathfrak l$, $\mathfrak K\cap H$ contains no planes, but if one intersects it with a three-hyperplane inside the four-hyperplane $H$, the intersection is a quadric surface. An easy geometric argument shows that any three-hyperplane inside $H$ can be put into a four-hyperplane $H'=T_{\mathfrak l}\mathfrak K$, that is $H'$ is tangent to $\mathfrak K$ at some point $\mathfrak l$. Thus the physical points $q$ and planes $\pi$, such that the corresponding $\alpha$ and $\beta$-lines in the phase space are contained in the two-quadric $\mathfrak K\cap H\cap H'$ are exactly those points and planes, incident in the physical space to the line $l$ -- the Klein map pre-image of $\mathfrak l$. This accounts for the role of the parameter $k$ in Theorem \ref{mish} versus Theorem \ref{gkt}.

With this construction in mind, the proof of Theorem \ref{mish} becomes mostly a technical matter, given the proof of Theorem \ref{gkt} in \cite{GK} and the fact that the latter works over a general $F$ if $\min(|Q|,|\Pi|)<p^2$ in positive characteristic. The origin of the constraint is the applicability of the Monge-Salmon theorem, bounding the number of lines that a non-ruled irreducible algebraic surface in $\Pro^3$ may support, in terms of the degree $D>2$ of the surface, {\em provided that}  $D<p$. See  \cite{Vo}, \cite{ES}, \cite{Ko} and \cite{Rud} for details. 

De Zeeuw \cite{dZ} developed a ``physical space'' proof of Theorem \ref{mish} that requires no familiarity with the Klein quadric and its rulings by planes and can therefore be presented more economically. We briefly describe it in the language of the above presentation. While in \cite{Rud} the hyperplane $H$ was chosen to intersect $\mathfrak K$ transversely, it can, in fact, be chosen as $H=T_{\mathfrak l}\mathfrak K$, the tangent space at a generic point $\mathfrak{l}\in \mathfrak K$. The variety $\mathfrak K\cap T_{\mathfrak l}\mathfrak K$ is a three-dimensional quadric, which physically corresponds to the set of all physical lines in $\Pro^3$ meeting the Klein pre-image $l$ of $\mathfrak l$. It is often called a {\em singular line complex}, versus a {\em regular} one arising when $H$ cuts $\mathfrak K$ transversely (which is geometrically different and enables a rather different interpretation in the physical space, see \cite[Section 4]{Rud} and more generally \cite[Chapter 3]{PW}).  

This underlies the following affine (rather than projective) parameterisation that led de Zeeuw \cite{dZ} straight to an application of a bi-partite version of Theorem \ref{gkt}, presenting which he took some shortcuts, referring to the paper of Koll\'ar \cite{Ko}. Choose a generic (neither containing any $q\in \Q$, nor itself contained in or parallel to any $\pi\in \Pi$) affine line $l_0\subset F^3$ and a generic (thus not containing either any $q\in \Q$, or $l_0$) affine plane $\pi_0\subset F^3$.   Fix affine coordinate systems $z\in F$ on $l_0$ and $(x,y)\in F^2$ on $\pi_0$, with the dual coordinates $(x^*,y^*)$.  For $q\in Q$, consider a pencil of lines incident to both $q$ and $l_0$. Parameterise the pencil by the one-dimensional coordinate $z\in l_0$ of the intersection of a line in the pencil with $l_0$, as well as the pair $(x,y)$ of the intersection of this line in the pencil with $\pi_0$. It's easy to see that the affine pencil becomes parameterised as a line $(x(q,z),y(q,z),z)\subset F^3$. Furthermore, each plane $\pi\in \Pi$ will intersect $l_0$ at a point $z=z_\pi$, and a line in the pencil of lines in $\pi$ through $l_0\cap \pi$ gets parameterised as $(x^*(z_\pi),y^*(z_\pi), z_\pi)$.

\medskip
We end this discussion by mentioning that Theorem \ref{mish} has a corollary of independent interest, concerning the number of incidences between a set $L$ of lines and set $P$ of points in a projective three-quadric $H\cap \mathfrak K$,  where the hyperplane $H$ in the phase space intersects $\mathfrak K$ transversely. The affine part of this quadric can be viewed as $SL_2(F)$ with its standard embedding in $F^4$. Lines in $SL_2$ are cosets of one-dimensional subgroups, conjugate to  $\left\{\left(\begin{array}{ll}1&t\\0&1\end{array}\right),\;t\in F\right\}$.  

\begin{corollary}\label{sl2lines} Consider a finite set $L$ of affine lines in $SL_2(F)\subset F^4$, with $|L|<p^2$ if $F$ has characteristic $p>0$, suppose at most $k$ lines lie in a two-quadric. The set of incidences $I(P,L)$ of $L$ with a finite set of points $P\subset SL_2(F)$  satisfies the bound
$$
|I(P,L)| \ll |P|^{\frac{1}{2}}|L|^{\frac{1}{2}}(|L|^{\frac{1}{4}}+k^{\frac{1}{2}}) + |P|.
$$
\end{corollary}

Corollary \ref{sl2lines} holds {\em without} the assumption that no more than two lines meet at a point. It becomes a restricted -- by the fact that  $L$ is a subset of a three, rather than four-dimensional variety of lines -- general $F$ version of the point-line incidence theorem in $\R^3$, cited as \cite[Theorem 1.1]{SS} implicit in \cite{GK}. It's worth pointing out that lines in $SL_2$, concurrent at some point, lie in a two-quadric.

\begin{proof} Use the same notations $(P,L)$  for corresponding pair of projective sets of points and lines in $H\cap \mathfrak K$. Each projective line in $L$, a physical line pencil in $\Pro^3$, lifts uniquely as a pair $(\alpha_q,\beta_\pi)$ of two-planes ruling $\mathfrak K$. Thus the set of lines $L$ produces the pair $(Q,\Pi)$ of point and plane sets in $\Pro^3$, both of cardinality $|L|$.
By the Cauchy-Schwarz inequality 
$$
|I(P,L)|\leq\sqrt{|P|}\sqrt{|I(Q,\Pi)|} + |P|
$$
and the claim follows by Theorem \ref{mish}.
\end{proof}

\section{On isotropic directions} \label{null} This short section contains the necessary minimum, concerning isotropic vectors in $F^d$, where $d=3,4.$ A nonzero vector $s\in F^d$, $d\geq 2$ is {\em isotropic}, or {\em null} if $\|s\|^2=s\cdot s=0$, relative to the standard dot product. (Throughout {\em orthogonality,} or {\em normality}, or {\em right angle} of vectors $s,t$ means that $s\cdot t=0.$) In $F^2$ there are no isotropic vectors if $-1$ is not a square in $F$ -- in the context of $F=\F_p$, this means  $p\equiv 3 \pmod 4$. Otherwise $F^2$ has a (non-orthogonal) basis of isotropic vectors $s=(1,\pm \i)$, where $\i^2=1$. 

In $F^3$ isotropic vectors form an {\em isotropic cone} $\mathcal S^2_0$ through the origin; in $\F_p$ it is the union of $p+1$ lines for odd $p$. 

By non-degeneracy of the dot product (that is if $W$ is a subspace of $F^d$, then the dimensions of $W$ and its orthogonal complement add up to $d$) if $s,t$ are nonzero isotropic vectors in $F^3$, with $s\cdot t=0$, then one is a scalar multiple of the other. Therefore, the only nontrivial {\em null triangles} in $F^3$,  that is triangles $rst$ whose all three sides are {\em null pairs} (that is $r-s,\,s-t,\,t-r$ are all isotropic vectors) are degenerate ones, namely when the three vertices $r,s,t$ lie on some isotropic line. 

Indeed, otherwise, from
$$
s-t = (s-r) + (r-t),
$$
hence $(s-r) \cdot (r-t) =0$, one deduces that plane in $F^3$, defined by the triangle $rst$ is fully isotropic, that is has an orthogonal basis of isotropic vectors and hence is contained in its orthogonal complement, which contradictions non-degeneracy of the dot product. 

For similar arguments in the same vein see, e.g., \cite[Lemma 5.1]{HI}; note that the proof there does not work in $\F_3$,  for the same reason that the presented sketch of proof of the forthcoming Lemma \ref{l:null} is vacuous in $\F_3$.

If $s$ is an isotropic vector in $F^3$ ($F^4$), we refer to its orthogonal complement, the plane (hyperplane) $s^\perp$, as a {\em semi-isotropic} plane (hyperplane). Moreover, $F^4$ contains isotropic, or {\em fully isotropic} planes, spanned by a pair of mutually orthogonal isotropic vectors.

For $t\neq 0$  and $d\geq 2$ the sphere $\mathcal S^{d-1}_t\subset F^d$ is defined as
$$
\mathcal S^{d-1}_t=\{x:\,x_1^2+\ldots+x_{d}^2 =t\}.
$$
It is easy to verify the following statement. 

\begin{lemma} \label{l:null} For $t\neq 0,$ the two-sphere $\mathcal S^2_t$ is doubly ruled by lines if  $-t$ is a square in $F$, otherwise $\mathcal S^2_t$ contains no lines.
\end{lemma}

\begin{proof}[Sketch of proof] We consider only the generic $x=(x_1,x_2,x_3)\in \mathcal S^2_t$, whose all components are nonzero, and on top of this $x_2^2+x_3^2 \neq 0$. Otherwise, one has to chase through a few special cases, leading to the same conclusion.  (Note that   if $F=\F_3$ there is no such $x$, but the claim of the lemma is easily verified by hand for $t=\pm1$.) Suppose, an isotropic vector $d=(1,\alpha,\beta)$ is orthogonal to $x$.
 
Then by orthogonality $\beta = -\frac{x_1+\alpha x_2}{x_3}$. Since $d$ is an isotropic vector, it follows that
$$
\alpha^2+ 2\frac{x_1x_2}{x_2^2+x_3^2}\alpha + \frac{x_1^2+x_3^2}{x_2^2+x_3^2}  = 0.
$$
Equivalently,
$$
\left(\alpha + \frac{x_1x_2}{x_2^2+x_3^2}\right)^2 = -  \frac{x_3^2(x_1^2+x_2^2+x_3^2)} {(x_2^2+x_3^2)^2} =  \frac{x_3^2} {(x_2^2+x_3^2)^2} \cdot (-t)\,. 
$$
The statement follows, since if $-t$ is a square there are two roots of the latter equation in $\alpha$,  and none otherwise. \end{proof}

The sphere $\mathcal S^3_t$, on the other hand, intersects its tangent space at a point $x$ ($x$ itself is not isotropic) along a two-dimensional cone, formed by isotropic vectors orthogonal to $x$.

We will deal with isotropic lines in $F^3$ or a three-quadric $\mathcal S^3_t$, and use the following lemma. 
\begin{lemma} \label{easy} 
For a finite set $A\subset F^3$ or $A\subset \mathcal S^3_t$,  either $\Omega(|A|)$ points are collinear on a isotropic line, or a positive proportion of $A\times A$ are not null pairs. 
\end{lemma}
\begin{proof}
Build a graph $G$ on the vertex set $A$, connecting vertices $(a, b)$ by an edge if $(a,b)$ is a null pair, and we shall show that $G$ needs $\gg |A|^2$ additional edges to be turned into a complete graph.
	
Suppose, $A\subset F^3$. If  $|A|\geq Cp$, for some sufficiently large absolute constant $C$, then the number of collinear point triples in $A$, lying on isotropic lines is at most $|A|p(p+1)\ll \frac{|A|^3}{C^2}$. A collinear triple on an isotropic line is the only way to have a triangle in the graph $G$, so assuming that just a sufficiently small proportion of $A$ may lie on an isotropic line implies that $G$ has small triangle density. On the other hand, if  the edge density of $G$ were as large as $1-\epsilon$, for a sufficiently small $\epsilon>0$ -- which would mean that $ a - b$ were non-isotropic for only $\ll \epsilon |A|^2$ pairs $(a, b)$ -- then the triangle density would be at least $1-O(\epsilon)$. For this claim one needs merely the pigeonhole principle, or use as a black box a much more fine-tuned asymptotic formula by Razborov \cite{Raz}, which tells that that the triangle density of a graph with edge density $1-\epsilon$ is at least $1-3\epsilon + O(\epsilon^2)$.

The same proof applies to $A\subset \mathcal S_t^3$, because $S_t^3$, being three-dimensional, also cannot contain a fully isotropic two-plane. Indeed, one cannot have two distinct mutually orthogonal isotropic lines tangent to $\mathcal S^3_t$ at some point $z$, for these two lines would span a fully isotropic self-orthogonal plane, which is also orthogonal to the non-isotropic vector $z$. This contradicts non-degeneracy of the dot product: $2+3\neq4$.
\end{proof}

\begin{remark}\label{3sph} We remark that the intersection of $S_t^3$ with a fully isotropic plane is one isotropic line. Indeed, if $u,v$ are isotropic mutually perpendicular vectors in $F^4$ and $x\cdot x=t$, then for scalars $(\alpha, \beta)$, the condition $x+\alpha u +\beta v\in S_t^3$ defines a nontrivial linear equation in $(\alpha, \beta)$.

Besides, suppose $l \subset S_t^3$ is an isotropic line via $x$. Then $l^\perp \cap S_t^3$ is a developable quadric: a cylinder of isotropic lines parallel to $l$. Indeed, let $u$ be a fixed vector in the direction of $l$ and $v$ in some direction orthogonal to $u$, with $v\cdot v\neq 0$. The condition $x+\alpha u +\beta v\in S_t^3$ now reads $\beta=0$ or $\beta=\beta(v) = - 2\frac{x\cdot v}{v\cdot v}$, so $x+\beta v\in S_t^3$. The family of admissible values of $\beta$ is one-dimensional, hence $l^\perp \cap S_t^3$ is as described.
\end{remark}

\begin{remark} Since over an algebraically close field $F$, the three-quadrics $SL_2$ and $\mathcal S^3_1$ are projectively equivalent, Corollary \ref{sl2lines} also applies as a point-line incidence bound involving a set of (isotropic) lines $L$ in the unit sphere $\mathcal S^3_1$.
\end{remark}

\section{Applications to Erd\H os-type geometric questions} \label{erdos}
This section has two parts. First, we develop an application of Theorem \ref{mish} to the problem of counting distinct values of a non-degenerate bilinear form on pairs of points lying in a plane point set, extending to positive characteristic the estimates  one easily obtains over $\R$ (as well as $\C$, for it applies there as well \cite{T}) via the Szemer\'edi-Trotter theorem. Then we  consider the positive characteristic  version of the Erd\H os distance problem in dimensions three and two. 

Throughout this section, $F$ is a field of positive odd characteristic $p$.


\subsection{On distinct values of bilinear forms}\label{forms}

The  challenge  of  getting the best possible lower bounds for the cardinality of the set $\omega(S)$ of values of a non-degenerate bilinear form $\omega$, evaluated on pairs of points from a finite set $S\subset F^2$ of  points in the plane has historically received much less attention than the renown and at the first sight similar question of Erd\H os \cite{Erd}  about the number of distinct distances defined by  $S$ in the real plane, which was resolved by Guth and Katz \cite{GK}. 

However, if $\omega$ is symmetric, say the standard dot product, the author is unaware of a better proven bound than $|\omega(S)|\gg |S|^{\frac{2}{3}}$,  even over the reals, where it follows immediately after bounding as $O(|S|^{4/3})$ the maximum number of realisations of a single  nonzero value of $\omega$ via the Szemer\'edi-Trotter theorem. This also evinces dissimilarity between the two questions, for the upper bond $O(|S|^{4/3})$ on the maximum number of realisations of  a single value of $\omega$ is tight, while the number of realisations of a single distance is believed to be $\lesssim |S|$, constituting the Erd\H os single distance conjecture. Moreover, distance sets are left invariant by the three-dimensional Euclidean group, while the dot products -- only by the one-dimensional Orthogonal group. 

This is what makes a skew-symmetric $\omega$ special, for the set $\omega(S)$ is invariant to $SL_2$-action on vectors in the plane as matrix multiplication.

 One may conjecture -- clearly, in positive characteristic this may generally hold only if  $|S|<p$ -- 
 that as long as $\omega(S)\setminus\{0\}$ is nonempty (which is from now on implicit) then  $|\omega(S)|\gtrsim |S|, $ although in contrast to the case of distances, the author is unaware of examples, where $0<|\omega(S)|=o(|S|)$. 
The problem was claimed to have been solved over $\R$  in \cite{IRR}, $\omega$ being the cross or dot product. However, the set-up of the proof was flawed. The error came down to ignoring the presence of nontrivial weights (multiplicities), as they appear below. The best bound over $\R$  that the erratum \cite{IRRE} sets, is $|\omega(S)| \gg |S|^{\frac{96}{137}}$, for a skew-symmetric $\omega$. In positive characteristic,  if $|S|\leq p^{\frac{162}{161}}$, the best bound is $|\omega(S)| \gtrsim |S|^{\frac{108}{161}}$. Note that $\frac{108}{161}=\frac{2}{3}+\frac{2}{483},$ so  the state-of-the-art positive characteristic bound for a skew-symmetric $\omega$ is better than the known Euclidean bound for a symmetric $\omega$ .

 \medskip
This section presents the following theorem, slightly generalising \cite[Theorem 13]{Rud}.
\begin{theorem}\label{spr}
Let $\omega$ be a non-degenerate  bilinear form, the set $S\subseteq F^2$ and $\omega(S)\neq \{0\}$.
Then
\begin{equation}|\omega(S)| \gg \min\left(|S|^{\frac{2}{3}},p\right).\label{worst}\end{equation}
If $S'\subseteq S$ is a maximum subset of points, all lying in distinct directions  through the origin, then  $|\omega(S')|\gg \min( |S'|, p)$.
\end{theorem}

\begin{proof}  
From now on we assume that there are no ``very rich'' lines $l$ through the origin in $F^2$, that is lines supporting more than $|S|^{\frac{2}{3}}$ points of $S$, for otherwise one gets more than $|S|^{\frac{2}{3}}$ distinct values of $\omega(s,s')$ for $s'\in l$ and some $s\in S$. This trivial argument justifies estimate \eqref{worst} if there exists a very rich line. The claim concerning the subset $S'$ follows from the forthcoming argument independently.

Furthermore, since $\omega$ is non-degenerate, the problem is equivalent to asking for a lower bound for cardinality of the set
$$S\wedge T :=  \{s\wedge t:\,s\in S,\,t\in T=A(S)\}|\,,$$ where $A$ is a linear isomorphism and $\wedge$ the standard wedge product.

Consider the equation
\begin{equation}\label{eng} s \wedge t =s'\wedge t'\neq 0\,: \qquad (s,s',t,t')\in  S\times S\times T\times T\end{equation}
and rewrite it as
\begin{equation}\label{engg}
s \wedge t + t'\wedge s' =0\,.
\end{equation}
The latter equation can be viewed as counting the number of incidences between the set of points $Q\subset \Pro^3$ with elements $(s_1:s_2:t'_1:t'_2):=(s:t')$, written in homogeneous coordinates, and planes in a set $\Pi$ defined by covectors $(t_2:-t_1:s_2':-s_1'):=(t^\perp:{s'}^\perp)$. However, both points and planes are weighted. Namely, the weight 
$w(q)$ of a point $q=(s:t')$ is the number of points $(s,t')\in F^4$, which are projectively equivalent. Geometrically the equivalence class of $(s:t')$  contains all pairs of points in $S\times T$, which are obtained from $s$ and $t'$, respectively, via a homothety (dilation through the origin). Similarly, planes also carry weights. The total weight of both sets of points and planes is $W=|S|^2$. 

Note, however, that once we restrict $S$ to the subset  $S'\subset S$ as in the statement of the theorem, there are no weights exceeding $1$.
For $S$ itself, the maximum weight $w_0$ is trivially bounded by the maximum number of points in $S$ through the origin, i.e.,
$$
w_0\leq |S|^{\frac{2}{3}}.
$$
On the other hand, decreasing $w_0$ if necessary, we can fix a sufficiently small absolute $\epsilon>0$ and assume that a positive proportion of the set $S$ is supported on lines through the origin, each having the number of points of $S$ in the interval $[w_0^{1-\epsilon},w_0]$. We restrict $S$ to these points, as well as its linear image $T$ and the weighted sets $Q,\Pi$ of points and planes in $\Pro^3$, retaining the notations $S,T,Q,\Pi$, but  bearing in mind that the constants hidden in $\ll$ symbols now depend on powers of $\epsilon$.

Thus the number of solutions of \eqref{engg}, {\em including} the quadruples yielding zero values of $\omega$ is the number of weighted incidences
\begin{equation}\label{weightin}
I_w := \sum_{q\in Q, \pi\in \Pi} w(q)w(\pi) \delta_{q\pi},
\end{equation}
where $\delta_{q\pi}$ is $1$ when $q\in \pi$ and zero otherwise.

To estimate the latter quantity we use Theorem \ref{wmish}. Let us show that if we ignore point-plane incidences in $I(Q,\Pi)$ along a set $L^*$ of forbidden lines to be described -- these incidences being in  correspondence with zero values of the wedge product in \eqref{eng} -- the maximum number $k^*$ of collinear points in $Q\subset\Pro^3$ in  the application of Theorem \ref{wmish} can be bounded by $|S|w_0^{\epsilon-1}$. 

Consider a two-plane through the origin in $F^4=F^2\times F^2$, supporting some  points of $S\times T$. The number of points of $Q$ on the corresponding line in $\Pro^3$ is the number of lines through the origin, supporting points of $S\times T$ and lying  in the  above two-plane through the origin in $F^4$. If the two-plane projects on one of the coordinate planes $F^2\times F^2$ one-to-one, the number of such lines is at most $|S|w_0^{\epsilon-1}$. Otherwise, we forbid the two-plane, that is the corresponding line in $\Pro^3$ is added to the forbidden set $L^*$.
 
This happens if and only if the two-plane in question is the Cartesian product $l_1\times l_2\subset F^2\times F^2$, where $l_1,l_2$ are two lines through the origin, containing between $\ceil{w_0^{1-\epsilon}}$ and $w_0$ points of $S$ and $T$, respectively, each. So the $L^*$ of forbidden lines in $\Pro^3$ is defined by two-planes $l_1\times l_2\subset F^2\times F^2$, such that $l_1,l_2$ support the number of points of $S$ and $T$, respectively, in the interval  $[w_0^{1-\epsilon},w_0]$.

A point-plane incidence along a forbidden line is the solution of equation \eqref{eng}, where $s\in l_1$, $t'\in l_2$. As far as the quantities $s',t$ are concerned, the corresponding forbidden line also lies in a (projective) two-plane $\pi\subset \Pro^3$, defined by the homogeneous coordinate covector $(t^\perp:{s'}^\perp)$. Thus both lines $l_1\in F^2\times\{(0,0)\}$ and $l_2\in \{(0,0)\}\times F^2$ must lie in the three-hyperplane in $F^4$, identified by the covector $(t^\perp:{s'}^\perp)$. Clearly, this happens if and only if $t\in l_1$ and $s'\in l_2$. Returning to equation \eqref{eng} one sees that these incidences correspond to the zero value of the wedge product, which is not being counted.

Hence we can apply the $I_w^*$-version of estimate \eqref{pupsweight}, Theorem \ref{wmish},  obtaining
\begin{equation}\label{weste}
I^*_w\ll |S|^2 \left(|S|\sqrt{w_0} + |S|w_0^{\epsilon}\right)\,.
\end{equation}
The bound is valid as long as $|S|\leq p^{\frac{3}{2}}$. Otherwise one may restrict $S$ to a subset of cardinality $p^{\frac{3}{2}}$ at the outset. Estimate \eqref{worst} now follows from the standard application of the Cauchy-Schwarz inequality, using the latter bound for the number of solutions of equation \eqref{eng}, that is 
$$
|\omega(S)| \geq C(\epsilon) \frac{|S|^4}{|S|^{\frac{10}{3}}}\,,$$
with some absolute constant $C(\epsilon)$. 

Finally, if the set $S$ is replaced by $S'$ as in the formulation of the theorem, this means applying to estimate \eqref{eng} the non-weighted incidence bound \eqref{pups}, with $k=1$ or equivalently having \eqref{weste} with $w_0=1$ and $\epsilon=0$. 

\end{proof}

\subsection{On distinct distances}
A well-known question of Erd\H os about distinct distances \cite{Erd}, can be generalised as follows. For a finite $S\subset F^d$, $d>1$, define
the ``distance'' set
$$\Delta(S) = \{\|s-t\|^2:\,s,t \in S\},$$
where $\|s\|^2 = s\cdot s=\sum_{i-1}^d s_i^2.$

If $F=\R$ , $d=2$, the question was resolved by Guth and Katz \cite{GK} (the proof also applies to the spherical and hyperbolic distance \cite{RS}) and  is open in dimension three and higher. In  $\R^3$ the conjecture claims that there are $\gtrsim |S|^{\frac{2}{3}}$ distinct distances.  The best known bound is $\Omega(|S|^{.5643})$, due to Solymosi and Vu \cite{SV}.

\subsubsection{Distances in $F^3$}

We prove the bound $\Omega(\sqrt{|S|})$ for the positive characteristic pinned version of the problem, i.e., for the number of distinct distances, attained from some point $s\in S$, for $|S|\leq p^2$, assuming that $S$ is not contained in a single semi-isotropic plane, where the number of distinct distances can be smaller. In a semi-isotropic plane $y^\perp$, for an isotropic $y$, with an orthogonal basis $\{x,y\}$, one can have a set $S$ of $|S|=kl$ points, with $1\leq k\leq l$ and just $O(k)$ distinct pairwise distances: place $l$ points on anywhere on each of $k$ parallel lines in the direction $y$, whose $x$-coordinates are an interval $[1,\ldots,k]$.

In $\R^3$, the bound $\Omega(\sqrt{|S|})$, being a kind of threshold one  for the number of distinct pinned distances, was established in 1990 in the milestone paper  by Clarkson et al. \cite{CEGSW}. The proof was partially based on a space partitioning technique that the real setting enables. More recent stronger bounds, e.g., the above-mentioned bound by Solymosi and Vu \cite{SV} or Zahl's  \cite{Z} result apropos of the single distance rely on perfecting the partitioning techniques.

Presented next is an easy partition-free proof of the threshold  estimate $\Omega(\sqrt{|S})$ for the number of distinct distances, as an application of Theorem \ref{mish}.

\begin{theorem}\label{erd} A finite set $S\subset F^3$, not supported in a single semi-isotropic plane, determines $\Omega[\min(\sqrt{|S|},p)]$ distinct pinned distances, i.e., distances from some $s\in S$ to points of $S$.
\end{theorem}

\begin{proof}
First off, let us restrict $S$, if necessary, to a subset of at most $p^2$ points. Furthermore, assume that there are at most $\sqrt{|S|}$ collinear points or there is nothing to prove. Indeed, even if the line supporting $\sqrt{|S|}$  points is isotropic (otherwise the claim is trivial) $S$ has another point $s$  outside this line, such that the plane containing $s$ and the line is not semi-isotropic, and  then there are $\Omega(\sqrt{|S|})$ distinct distances from $s$ to the points on the  line.

Define
\begin{equation}\label{energyd} 
\E_\Delta:= |\{(s,t,t')\in S\times S\times S:\,\|s-t\|^2 = \|s-t'\|^2\neq 0\}\,. \end{equation}
The claim of the theorem will  follow if we establish that either $S$ contains a line with $\Omega(\sqrt{|S|})$ points or
\begin{equation}\label{clm}
\E_\Delta =O(|S|^{\frac{5}{2}})\,.
\end{equation}
The former case has been addressed, in the latter case an application of the Cauchy-Schwarz inequality will do the job., 

By the pigeonhole principle and  Lemma \ref{easy}, assuming $\E_\Delta\gg |S|^{\frac{5}{2}}$ implies that either there is a isotropic line with $\Omega(\sqrt{|S|})$ points, or $\E_\Delta=O(\E_\Delta^*)$, where $\E_\Delta^*$ is the number of solutions of the equation
\begin{equation}\label{energystar} 
\|s-t\|^2 = \|s-t'\|^2\neq 0\,,\qquad (s,t,t')\in S\times S\times S:\;\|t-t'\|\neq 0\,.
\end{equation}

Indeed, the quantity $\E_\Delta$ counts the number of equidistant pairs of points from each $s\in S$ and sums over $s$. If $\E_\Delta\gg |S|^{\frac{5}{2}}$, a positive proportion of  $\E_\Delta$ is contributed by points $s$ and level sets $Z_r(s)=\{t\in \F^3:\, \|s-t\|^2=r\}$, such that $Z_r(s)$ supports $\Omega(\sqrt{|S|})$ points of $S$. 
By Lemma \ref{easy} either there is a line with $\Omega(\sqrt{|S|})$ points, or  a positive proportion of pairs of distinct  $t,t'\in Z_r(s)$ are non-null.

It therefore remains to justify the bound 
\begin{equation}\label{energystarr} \E_\Delta^*\ll |S|^{\frac{5}{2}}\,\end{equation}
assuming that no line supports more than $\sqrt{|S|}$ points of $S$.

\medskip
To evaluate the quantity $\E_\Delta^*$: for each pair $(t,t')$ -- which is not a null pair -- we have a plane through the midpoint of the segment $[t\,t']$, normal to the vector $t-t'$ and need to count points $s$ incident to this plane. The plane in question does not contain $t$ or $t'$.

We arrive at an incidence problem $(S,\Pi)$ between $|S|$ points and a multiset of planes, defined by non-null pairs $(t,t')$.  The same plane can bisect up to $|S|/2$ segments $[t \,t']$, for there is at most one $t'$ for each $t$, such that the plane may bisect  $[t \,t']$. We would have been done earlier, unless the number of distinct planes is $\gg|S|$, and therefore bound \eqref{energystarr} follows from Theorem \ref{mish}.

Theorem \ref{erd} follows from \eqref{energyd} by the Cauchy-Schwarz inequality.  In particular, when $|S|=p^2$, one gets $\Omega(p)$ distinct pinned distances. 
\end{proof}

\subsubsection{Distances in $F^2$} The argument in two dimensions is similar and uses  the point-line incidence bound  in Theorem \ref{SdZgen} instead of Theorem \ref{mish}. 

 \begin{theorem}[\cite{SdZ}, Corollary 13]
A finite set $S\subset F^2$, not supported on a single isotropic line and such that $|S|\leq p^{\frac{15}{11}}$, determines $\Omega(| S|^{\frac{8}{15}})$ distinct distances from some $s\in S$.
\label{2ddist}\end{theorem}
If $F=\F_p$, there is also a large set case estimate, which soon takes over when $|S|$ exceeds $p$. See \cite[Proof of Theorem 2.2]{Io}, \cite[Proof of Theorem 1.6]{mamma}, namely
$$
|\Delta(S)|\gg   \frac{p}{1+ p^2|S|^{-\frac{3}{2}}}\,,
$$
which,  although vacuous for $|S|\leq p$, beats the claim of Theorem \ref{2ddist} already for $|S|\geq p^{\frac{30}{29}}$.  The key quantity behind the latter estimate is the energy-type non-pinned version of equation \eqref{energystar}, that is the variable $s$ in the right-hand side of the three-variable equation in \eqref{energystar} turns into the fourth variable $s'\in S$. For the minimum number of  pinned distances for $|S|\geq p^{\frac{15}{14}}$ in $\F_p$ one can use estimate \eqref{insv}, which yields the existence of $\Omega\left(\frac{p}{1+ p^{\frac{3}{2}}|S|^{-1}}\right)$ distinct distances from some $s\in S$.

\begin{proof}[Sketch of proof of Theorem \ref{2ddist}] Consider equation \eqref{energy}, assuming that at most $\epsilon|S|$ points of $S$ are collinear, for some absolute $\epsilon>0$, or there is nothing to prove. It follows that most pairs $(t,t')$ are not null, and the estimate for the number of solutions of \eqref{energy} is tantamount to estimating the number of incidences of $|S|$ points and a multiset of $\Omega(|S|)$ distinct lines, the total weight of the set of lines being bounded by $|S|^2$. In the worst possible case there are $|S|$ lines with maximum weight $|S|$, to which one applies the first estimate  in \eqref{ins}. The claim of  Theorem \ref{2ddist} follows from the latter incidence bound by the Cauchy-Schwarz inequality.
\end{proof}

\begin{remark} As this manuscript was being prepared, a better bound $\Omega(|S|^{\frac{1128}{2107}})$ if $|S|<p^{\frac{7}{6}}$ for the number of distinct distances defined by a non-isotropic-collinear point set  $S\subset \F_p^2$ was proved by Iosevich et al \cite{IVN}. The improvement is based on the new observation that if there is a line in the proof of Theorem \ref{2ddist}, incident to a large number of points, one can consider distances between points on the line and the rest of $S$. Dealing with the latter distances enables one to take advantage of Theorem \ref{l:corners}, presented in the sequel here. In effect,  \cite{IVN} succeeds in using the incidence bound of Theorem \ref{SdZgen} twice, rather than once.\end{remark}

\section{Additive energy on quadrics} \label{en}

This section discusses some  applications of Theorems \ref{mish} and \ref{SdZgen}, motivated by questions in Fourier analysis. These are geometric incidence applications that constitute its immediate focus, with just a superficial account of how it comes about that Fourier analysis questions get converted to additive energy estimates for sets supported on varieties. Within the scope of this paper, {\em varieties} are limited to two quadrics: the paraboloid and the sphere. Note that Lemma \ref{l:null} evinces that there are two distinct geometric types of  ``the sphere''  $\mathcal S^2_t\subset \F_p^3$. Throughout this section $F=\F_p$.

For finite sets $A, B$ in an abelian group, the energy is defined as
\begin{equation}\label{energy}
\E(A,B) = |\{(x,y,z,u)\in A\times B\times A\times B:\,x+y=z+u\}|\,,\end{equation}
with a shortcut $\E(A)=\E(A,A).$

As a motivation from Fourier analysis we mention that the Fourier approach to the Erd\H os distance problem and its generalisations call for estimating the so-called {\em spherical average}, namely the $L^2$ norm of the restriction of the Fourier transform of the characteristic function of the point set in question on each sphere, centred at the origin \cite{IR}. In \cite{Io}, \cite{mamma} after applying the H\"older inequality the spherical average estimate was converted to an additive energy estimate on the sphere (which was trivial as the sphere was $\mathcal S^1_t$). 

It is easy to show -- see \eqref{f:inc_d-1} and \eqref{f:inc_d} below -- that if the set $A$ lies on the discrete paraboloid or sphere, then (up to a permutation of vertices) the energy equals, respectively, the number of rectangles formed in $F^{d-1}$ by the horizontal projections of the points  $x,y,z,u$ on the paraboloid or the rectangles with vertices $x,y,z,u$ on the sphere in $F^d$. We always assume that $x,y,z,u$ in \eqref{energy} are pair-wise distinct, for alternative scenarios contribute merely $O(|A|^2)$ to the energy. By a {\em rectangle} we mean a point quadruple, such that the dot products of adjacent difference vectors is zero at every vertex $x,y,z,u$ -- see \eqref{f:inc_d-1}, \eqref{f:inc_d}. The concept of {\em adjacent} vertices arises after rearranging equation \eqref{energy}. E.g., rewriting it as $x=z+u-y$, one concludes that vertices $z$ and $u$ are adjacent to $y$; this also enables one to distinguish the four {\em sides} from two {\em diagonals}.  A rectangle lies in a two-plane, and {\em plane} will always mean a (affine) two-plane in the sequel. The Euclidean shibboleth {\em the line containing a side} is used for a line supporting a pair of adjacent vertices of the rectangle.

We consider the dimension $d=3,4$. The rectangles we encounter are of three types (in $d=4$ we deal only with rectangles, whose four vertices lie on the sphere). The first type is {\em ordinary} rectangles, that is lines along all the four sides are non-isotropic. The opposite case is  {\em degenerate} rectangles, namely when both sides, adjacent to a vertex are isotropic vectors. This implies, both in $F^3$ and for $x,y,z,u\in \mathcal S_t^3$ that all the four vertices all lie on the same isotropic line. Even though in $F^4$ one can have (fully) isotropic planes, that is self-orthogonal planes, such a plane cannot be tangent to a sphere $\mathcal S^3_t$, $t\neq 0$, for otherwise $x$ itself would be isotropic, that is $t=0$.

Members of the third rectangle type to be dealt with are  {\em semi-degenerate} rectangles, namely when lines, containing one pair of opposite sides are isotropic, and for the other pair of sides -- non-isotropic.  Given the line $l$ containing the isotropic side of such a rectangle, the rectangle must lie in the unique semi-isotropic (hyper)plane $l^\perp$, containing $l$.


\medskip

We proceed to introduce the discrete paraboloid case. A more detailed account of the restriction problem thereon can be found in the author's paper with Shkredov \cite{RSh}.

Set $V=F^d$, and define the discrete paraboloid 
$$
	\mathcal{P}^{d-1} := \{ (\u x=(x_1,\ldots, x_{d-1}),\, \u x\cdot \u x) ~:~ \u x\in F^{d-1} \} = \{ (x_1,\ldots x_{d-1}, x^2_1 + \ldots +  x^2_{d-1}) \} \subset V \,.
$$
For a vector $x = (x_1,\dots, x_d)\in F^d$  write $x=(\u x,h)$, referring to $\u x$ and $h$ as, respectively, horizontal and vertical coordinates, and use $\u A$ to denote the horizontal projection of $A\subseteq \mathcal P^{d-1}$: $A$ is a graph over $\u A$, and $|\u A|=|A|.$

Note that  $\mathcal{P}^{d-1}$ contains  the isotropic cone $\mathcal S_0^{d-2}$ in the horizontal hyperplane $x_d=0$. Moreover, at every point $x=(\u x, \|\u x\|^2)$ it intersects its tangent space at $x$ in a two-dimensional cone that projects on the coordinates $\u x$ as a corresponding translate of $\mathcal S_0^{d-2}$. 

\medskip
The Fourier extension problem is bounding some Lebesgue norm on $F^d$, $d\geq 2$ of the inverse Fourier transform of a complex-valued function $f$ on some variety in the dual space. An equivalent question is bounding the norm of the restriction to a variety of the Fourier transform of a function $g$ on $F^d$ in terms of the norm of $g$. Overwhelmingly, ``some variety'' means an irreducible quadric.
If $F=\R$, the restriction problem has a reputed history, which is beyond the scope of this review. Since the 2000s, after having been set up by Mockenhaupt and Tao \cite{MoT}, the question has also been studied in the finite field setting. It is most approachable -- owing to the forthcoming Lemma \ref{l:g^-energy} -- if the quadric is $\P^{d-1}$.

\begin{remark} There is more to the questions discussed further than $\F_p$-versions of open questions in real harmonic analysis. In 2003 Bourgain \cite[Section 3]{Bo} showed how an energy estimate for a set on $\mathcal{P}^{2}$ can be recycled to yield an explicit  low-entropy {\em two-source extractor} for simulating independence in computer science. For more details and references see the recent work of Lewko \cite{Lewko} which specifically uses the energy estimates of \cite{RSh} presented below for this purpose. Moreover, as it has  already been mentioned, energy estimates on $\P^2$ were used by Iosevich et al  \cite{IVN} to get the best known result on the number of distinct distances for sufficiently small sets in $\F_p^2$.\end{remark}

Mockenhaupt and Tao \cite[implicit in proof of Theorem 6.2]{MoT} showed that one can express restriction estimates to $L^2(\mathcal P^{d-1})$ in terms of energy estimates for sets, supported on $\mathcal P^{d-1}$. 

For a set $S\subseteq V$, with the vertical coordinate $h\in F$, define $S_h \subseteq \mathcal{P}$ as the horizontal $h$-slice of $S$, lifted to $\mathcal P^{d-1}$, that is
$$
S_h := \{ (\u x,\u x\cdot \u x): \,(\u x,h)\in S\}\,.
$$
In this context let us present (a variant of) \cite[Lemma 2.1]{IKL}, where a thorough sketch of the proof, which uses the Plancherel identity, H\"older inequality and Gauss sums. See \cite[Corollary 25, Lemma 29]{Lewko} for more details and references as to the following lemma. 

\begin{lemma}
	\label{l:g^-energy}
	Let $g: \F_p^d \to \mathbb{C}$ be a function such that $\| g\|_\infty \leq 1$ on its support $S$. 
	Then for its Fourier transform $\hat{g}$ one has
	\begin{equation} \label{e:g^-energy}
	\| \hat{g} \|_{L^2 (\P^{d-1},d\sigma)} \ll |S|^{\frac{1}{2}} + |S|^{\frac{3}{8}} 
	p^{-\frac{d-2}{8}}
	\left( \sum_{h\in \F_p} \E^{\frac{1}{4}} (S_h) \right)^{\frac{1}{2}} \,.
	\end{equation}
\end{lemma}
Lemma \ref{l:g^-energy} bounds the $L^2$ norm of the restriction of the Fourier transform of a function $g$ on $V$ to $\P^{d-1}\subset V^*$, the dual space, in terms of energies of the horizontal slices of  the support of $g$, lifted on $\P^{d-1}\subset V$. Energy estimates presented below, after having been plugged into \eqref{e:g^-energy} and interpolated with the so-called Stein-Tomas type bounds for the Fourier restriction problem, enabled Shkredov and the author  \cite{RSh} to prove the best possible $L^2(\P^3)$ restriction estimate over $\F_p$. For more details see \cite{RSh} and the references contained therein. We now move on to the energy estimates per se.

\subsection{Discrete paraboloid}
	
We start out with an easy but important observation concerning \eqref{energy}, that if $x,y,z\in \P^{d-1}$, then $u=x-z+y$ is in $\P^{d-1}$, that is
$(\u x- \u z+\u y)\cdot (\u x- \u z+\u y) = \u x \cdot \u x - \u z \cdot \u z + \u y \cdot \u y$,
 if and only if 

\begin{equation}\label{f:inc_d-1}
	(\u x - \u z) \cdot (\u z - \u y) = 0\,.
\end{equation} 
The same relation holds for any of the four triples of adjacent variables in the definition of energy \eqref{energy}, yielding a simple geometric criterion: a quadruple $(x,y,z,u)\in \mathcal (\P^{d-1})^4$ satisfies \eqref{energy} if and only if 
$(\u x,\u y, \u z,\u u)\in (F^{d-1})^4$ is a {\em rectangle}, as opposed to generally being a parallelogram to form an additive quadruple just in $(F^{d-1})^4$. Namely at each vertex $\u x,\u y, \u z,\u u$, the 
dot products of adjacent difference vectors is zero. Note that condition \eqref{f:inc_d-1} should hold at every vertex of the rectangle, hence finding all solutions of the latter relation applies to each geometric rectangle at least four and at most sixteen times. In the sequel we will use this freedom to chose the convenient corner $(x,y,z)$ of the rectangle and the fact that the energy is bounded by a constant times the number of geometric rectangles.

\subsubsection{The case $d=4$}

In this case $\P^3$ meets a tangent hyperplane at any point in a two-dimensional cone, made of isotropic lines.
Clearly if $A$ is supported on one of the isotropic lines, $\E(A)$ can be as big as $|A|^3$ or if $|A|\geq p$, then $|\E(A)|\gg |A|p^2$. In general, there is the following theorem.

\begin{theorem}
\label{l:par_energy} 
	Let $A\subseteq \mathcal{P}^3$. 
	Then
\begin{equation}\label{f:E_par_lemma}
\E(A) \ll \frac{|A|^3}{p} + |A|^{\frac{5}{2}} + |A|k_0^2 \,.
\end{equation}
where $k_0$ is the maximum number of points of $A$ on an isotropic line.	

\end{theorem}

\begin{proof}

A solution of equation (\ref{f:inc_d-1}) with all variables in $A$ can be interpreted as a point-plane incidence in $F^3$: $x$ being the point and the plane $\pi$ being the one	passing through $\u z$ and with the normal vector $\u y- \u z$. In the sequel we assume that $\u y\neq \u z$, for otherwise we have $|A|^2$ trivial solutions to the energy equation.
	
	The set of planes is, in fact, a multiset, but  linearity of incidence estimate \eqref{pups} of Theorem \ref{mish} in $|\Pi|$ makes it still apply with $|\Pi|=|A|^2$ and $|Q|=|A|$ as long as the number of distinct planes is $\gg |A|$. But this follows from the pigeonhole principle and Lemma \ref{easy}, unless a large proportion of $A$ lies on some isotropic line, in which case the last term in estimate \eqref{f:E_par_lemma} will do the job after a straightforward iterative procedure of removing that part of $A$ and continuing, the details being left to the reader.

Otherwise we apply Theorem \ref{Misha+}  obtaining an intermediate bound   
$$
\E(A) \ll \frac{|A|^3}{p} + |A|^{\frac{5}{2}} + k|A|^2  \,,
$$
where $k$ is the maximum number of collinear points in $\u A$. Clearly, if $k\ll \sqrt{|A|}$, we are done.

Otherwise, we proceed as follows.  Let $L$ be the set of all lines in $F^3$, supporting, say $\geq 10 \sqrt{|A|}$ points of $\u A$. By excluding the incidences along the lines in $L$ in the application of Theorem \ref{Misha+}, we succeed in counting all the rectangles in $\u A$, such that the line, containing at least one side of the rectangle is not in $L$. The number $\E_1$ of such rectangles therefore obeys the bound

\begin{equation}\label{f:E_par1}
\E_1 \ll \frac{|A|^3}{p} + |A|^{\frac{5}{2}} \,.
\end{equation}

 Let $L'\subseteq L$ be the subset of  non-isotropic lines and $\u A'$ the subset of $\u A$, supported on the union of these lines, and $A'$ its lift on $\mathcal P$. By the exclusion-inclusion principle, $|L'|\ll  \sqrt{|A|}$. 
We now bound $\E(A',A)$. This count will include all ordinary and semi-degenerate rectangles, the non-isotropic line supporting one of whose sides contains $\geq 10 \sqrt{|A|}$ points of $\u A$. Let $\E_2\leq \E(A',A)$ be the number of such rectangles.

Let $\u A_l := l \cap \u A'$, for a line $l\in L'$, $A_l$ denoting the lift  of $\u A_l$ on $\mathcal P$.

It is easy to see that
$$\E(A_l,A) = |\{ ( x,  y,  z, u) \in A\times A_l\times A\times A_l:\;  x+  y =  z + u\}|\ll |A_l||A|\,,$$
for if we fix a diagonal, say $xy$ of a rectangle, whose one side lies on a non-isotropic line, this fixes the remaining two vertices, the Euclidean way. Indeed, since the line in question is non-isotropic, there is a unique $z$ on this line to satisfy the orthogonality condition \eqref{f:inc_d-1}.

It follows by two applications of the Cauchy-Schwarz inequality that
\begin{equation}\label{f:E_par2}
	\E_2 \le \left( \sum_{l\in L'} \E(A,A_l)^{\frac{1}{2}} \right)^{2} \ll |A| \left(\sum_{l\in L'} |A_l|^{\frac{1}{2}} \right)^{2}
		\le
			|A| |L'| |A'| \; \ll \; |A|^{\frac{5}{2}}.
\end{equation}

It remains to count rectangles, all whose sides lie  on isotropic lines in $L$. Since there are no distinct mutually perpendicular pairs of isotropic directions in $F^3$, such rectangles can only be degenerate. If $\E_3$ is the number of such rectangles,
and it's easy to see that $\E_3 \ll |A|k_0^2$, in the worst possible case of $|A|/k_0$ isotropic lines with $k_0$ points on each.

Combining this bound with bounds \eqref{f:E_par1}, \eqref{f:E_par2}, since $\E(A)\ll\E_1+\E_2+\E_3$, completes the proof of Theorem \ref{l:par_energy}. 
\end{proof}

\subsubsection{The case $d=3$.} If $d=3$, the energy estimate on $\P^2$ is as follows.

\begin{theorem}\label{l:corners}
For $A\subseteq \mathcal{P}^2,$ with at most $k_0$ points on an isotropic line. One has
\begin{equation}\label{f:E2}
\E(A) \;\ll \;|A|k_0^2+\left\{ \begin{array}{lll} |A|^{\frac{17}{7}}\,, &for & |A|<p^{\frac{26}{21}}\,,\\ \hfill \\
 \frac{|A|^3}{p}  + |A|^2\sqrt{p}\,.\end{array}\right.
\end{equation}\end{theorem}	
Observe that if $p= 3\pmod 4$, the paraboloid contains no lines, in which case one can set $k_0=2$. 
Otherwise, for $p\equiv  1 \pmod 4$, $\mathcal{P}^{2}$ is doubly ruled by isotropic lines, and then the energy estimate shall inevitably have the term $|A|k_0^2,$ where $k_0$ is the maximum number of points in the projection of $A$ on the first  two variables on an isotropic line.

\begin{proof} In light of what has just been said, it suffices to assume $p=3\pmod 4$, so $\P^2$ contains no lines.  We restate equation \eqref{f:inc_d-1}, aiming to bound the number of solutions of 
\begin{equation}
(\u x- \u z)\cdot(\u z- \u y)=0:\;\;\u x,\u y, \u z\in \u A\,,
\label{right}\end{equation} where $\u A$ is the projection of $A$ on the $(x_1,x_2)-$plane.  Let us set $|A|=|\u A|=n$.

Equation \eqref{right} is a well-known problem of counting the maximum number of right triangles with vertices in the plane point set $\u A$, which in the real case was given a sharp answer by Pach and Sharir \cite{PS} via the Szemer\'edi-Trotter theorem. Here we adapt the argument in order to use the Stevens-de Zeeuw incidence bound \eqref{ins} instead.

Note that estimate  \eqref{insv} provides a universal bound
\begin{equation}\label{uni}
\E(A)\ll \frac{|A|^3}{p} + |A|^2\sqrt{p}.
\end{equation}

Let us recast bound \eqref{ins} in the usual way, aiming at the cardinality $m_k$ of the set of $k$-rich lines, that is lines, supporting $\geq k$ points of a $n$-point set:
$$
m_k  \ll  \frac{n^{\frac{11}{4}}} {k^{\frac{15}{4}}} + \frac{n}{k} + \frac{n^{\frac{13}{2}}}{p^{\frac{15}{2}}}\,,
$$
The third term in the bound arises as the alternative to the constraint $n^{13}m_k^{-2}< p^{15}$ of Theorem \ref{SdZgen}.

One may loosen the latter bound by subsuming its last term in the increased middle one (clearly, $k\leq p$):

\begin{equation}\label{k-rich}
m_k\ll  \frac{n^{\frac{11}{4}}} {k^{\frac{15}{4}}} + \frac{n^{\frac{5}{4}}}{k}\,,\qquad for \;\;\; n<p^{\frac{26}{21}}\,.
	\end{equation}

Next we are going to show that the number $N$ of nontrivial solutions (that is with $\u x\neq \u z$ and $\u y\neq \u z$) of equation \eqref{right} satisfies the following bound:
\begin{equation}\label{thin}
for \;\;\; n< p^{\frac{26}{21}}, \qquad N\ll n^{2+\frac{3}{7}}.
\end{equation}

Assuming $n< p^{\frac{26}{21}}$, let us express the quantity $N$  as follows. For $\u z\in A$, define $L_z$ as the set of all the $p+1$ lines in $F^2$ incident to $z$. For any line $l$ in $F^2$, let $n(l)$ be the number of points of $\u A$, supported on $l$ {\em minus 1}. Then
\begin{equation}\label{f:N}
N = \sum_{\u z \in \u A}\, \sum_{l\in L_z} n(l) n(l^\perp),
\end{equation}
where $l^\perp$ is the line orthogonal to $l$.

Let us set up a cut-off value $k_*= n^{\frac{3}{7}}$ of $n(l)$ to be justified. Partition, for every $z$, the lines $l\in L_z$ to poor ones, that is those with $n(l) \leq k_*$, and otherwise rich. 
Accordingly partition $N=N_{poor} + N_{rich}$, where  the term $N_{poor}$ means that at least one of $l,l^\perp$ under summation in \eqref{f:N} is poor, hence the alternative $N_{rich}$ is when both $l,l^\perp$ are rich. Clearly
$$
N_{poor} \leq 2k_* n^2.
$$
Let us now bound $N_{rich}$. Observe that the two terms in estimate \eqref{k-rich} meet when 
 $k\sim n^{\frac{6}{11}}$. Let us call the lines with $n(l) \geq n^{\frac{6}{11}}$ {\em very rich} and partition
 $$N_{rich} = N_{very-rich}+N_{just-rich}\,,$$
 the first term corresponding, for each $\u z$, to the sub-sum, corresponding to the case when one of $l,l^\perp$ is very rich.  Then one can bound $N_{very-rich}$ trivially, using the second term in \eqref{k-rich} and dyadic summation in  $k\geq n^{\frac{6}{11}}$  as
\begin{equation}\label{e:vr}
N_{very-rich}\ll n \sum_{l:\,n(l)\geq n^{\frac{6}{11}}} n(l) \ll n^\frac{9}{4}\log n,
\end{equation}
 which is better than 
 \eqref{thin}. 
 Indeed, given a very-rich line $l$ we count all triangles with vertices $\u x, \u z,\u y$, such that $\u z\in l$ and $\u y$ is any point outside $l$; the two will determine the third vertex $\u x\in l$.  

What is left to consider is the case of the summation in \eqref{f:N} when both  $n^{\frac{3}{7}}\leq n(l),\,n(l^\perp)\leq n^{\frac{6}{11}}$. We apply Cauchy-Schwarz to obtain 
$$
N_{just-rich}  \leq \sum_{\u z \in \u A} \;\;\sum_{l\in L_z:\,p^{\frac{3}{7}}\leq n(l) \leq n^{\frac{6}{11} } } n^2(l) .$$
The expression in the right-hand side counts collinear triples of points in $\u A$ on rich, but not very rich lines. The number of such lines with $n(l)\sim k$, for the range of $n(l)$ in question is bounded by the first term in estimate \eqref{k-rich}.  Multiplication of the latter term by $k^3$ followed by dyadic summation in $k$ yields
$$
N_{just-rich} \ll n^{\frac{11}{4}}  k_*^{-\frac{3}{4}}\,,$$ optimising with $N_{poor}\leq n^2k_*$ justifies the choice of $k_*=n^{\frac{3}{7}}$, and proves \eqref{thin}. Together with the better bound \eqref{e:vr} this completes the proof of Theorem \ref{l:corners}.
\end{proof}

\subsection{Discrete sphere} We now consider the same problem on the discrete sphere $\mathcal S^{d-1}_t$, $t\neq 0$. The results here have not appeared elsewhere, however, they are somewhat similar to the discrete paraboloid case that has been adopted from \cite{RS}. They are not identical though, because the geometries of the two quadrics are different.

First off, if $x,y,z \in S^{d-1}_t$, then for 
$x+y-z$ also to lie in $S^{d-1}_t$ one needs
\begin{equation}
0=2t + 2x\cdot(y-z)-2y\cdot z= 2(z-x)\cdot(z-y)\,,
\label{f:inc_d}\end{equation}
so the energy count on $S^{d-1}_t$ is once again, the count of rectangles with vertices in $A\subseteq S^{d-1}_t$, the rectangles themselves living in $F^{d}$ rather than $F^{d-1}$ as it was in the paraboloid case, but instead constrained by the fact that their vertices lie in $S^{d-1}_t$.

We now consider $d=3,4$ and present the analogues of the results in the preceding section in reverse formation, starting from $d=3$.

\subsubsection{The case $d=3$.} The main result in this section is the following energy estimate, which is slightly worse than the one in Theorem \ref{l:corners}.
\begin{theorem}\label{s_corners}
Let $A\subseteq \mathcal{S}_t^2,$ with at most $k_0$ points on an isotropic line. Then 
\begin{equation}\label{f:E2s}
\E(A) \;\ll \;|A|k_0^2 + \left\{ \begin{array}{lll} |A|^{\frac{37}{15}}\,, &for & |A|< p^{\frac{15}{11}}\,,\\ \hfill \\
 \frac{|A|^3}{p}  + |A|^2\sqrt{p}\,.\end{array}\right.
\end{equation}\end{theorem}	
\begin{proof}
Observe that by Lemma \ref{l:null}, the sphere  $\mathcal{S}_t^2$ contains isotropic lines on only if $-t$ is  square in $F$. In this case, since $\mathcal{S}_t^2$ is doubly ruled, the contribution of fully degenerate rectangles into $\E(A)$ is at most $|A|k_0^2$. Moreover, consider an isotropic line $l \subset \mathcal{S}_t^2$ and some $x\in l$. Then $x$ must be orthogonal to $l$, and therefore the plane, containing the origin and $l$ is the semi-isotropic plane $l^\perp$. This plane will meet also $\mathcal{S}_t^2$ at another line $l^{\|}$, parallel to $l$ and the total number of semi-degenerate rectangles with two vertices on $l$ and two on $l^{\|}$, over all isotropic lines in $\mathcal{S}_t^2$ is again $O(|A|k_0^2)$. (In other words, the only plane that may intersect $\mathcal{S}_t^2$ at two parallel isotropic lines and allow for an orthogonal direction is the semi-isotropic plane through the origin that contains both lines.)

After that we may assume that  $\mathcal{S}_t^2$ contain no lines and apply Theorem \ref{SdZgen} in a way the Szemer\'edi-Trotter theorem was used by Appelbaum and Sharir \cite{AS1}. One fixes $z\in A$ and counts the maximum number of right triangles in $A$ with the vertex $z$. There are $|A|-1$ distinct directions from $z$ to other points of $A$, and we  assign to each direction a point-line pair at the plane at infinity. The point is the ideal point in the direction $zx$, and the line -- ideal points in the directions normal to $zx$. Having the right triangle $xzy$ means an incidence between the ideal point, corresponding to $y$ and a line, corresponding to $x$, as well as the other way around. The claim of Theorem \ref{s_corners} then follows after an application of Theorem \ref{SdZgen} for small $A$, complemented by estimate \eqref{insv}.
\end{proof}

\subsubsection{The case $d=4$} We finally give the cousin of Theorem \ref{l:par_energy}. 
\begin{theorem}
\label{l:sph_energy} 
	Let $A\subseteq \mathcal{S}^3_t$, $t\neq 0$. 
	Then
\begin{equation}\label{f:E_sph_lemma}
\E(A) \ll \frac{|A|^3}{p} + |A|^{\frac{5}{2}} + |A|k_0^2 + |A|^2k_0 \,.
\end{equation}
where $k_0$ is the maximum number of points of $A$ on an isotropic line.	
\end{theorem}
Note that in view of Remark \ref{3sph} the last term $|A|^2k_0$ in estimate \eqref{f:E_sph_lemma}, which appears in addition to the estimate of Theorem \ref{l:par_energy} for $\P^3$, is unavoidable (as is the penultimate one, for the same reason as in Theorem \ref{l:par_energy}). Indeed, one can take some isotropic line $l\subset \mathcal S^3$ and consider the cylinder $l^\perp \cap\mathcal S^3$, in which one arranges $k_0$ points on each of $\frac{|A|}{k_0}$ lines, parallel to $l$. Then any parallelogram with two vertices on one line and two on another is a rectangle, so the number of the rectangles can be $\gg \left(\frac{|A|}{k_0}\right)^2 \cdot k^3$.

\begin{proof}
The aim of the proof is to give an upper bound on the number of rectangles $xzyu$, with vertices in $A\subset \mathcal S^3_t$. Let us split the rectangles into two main types. The first type is when a rectangle contains a side, say $zx$, which as a line is non-isotropic, in particular this line intersects $\mathcal S^3_t$ only at $\{x,z\}$. Otherwise the rectangle is of the second type. 

Following the notation in the proof of Theorem \ref{l:par_energy}, let $\E_3$ be the number of rectangles of the second type. This, according to Remark \ref{3sph}, means that the rectangle is degenerate, that is all its four vertices lie on a single isotropic line. To bound the number of degenerate rectangles we only need a trivial estimate for the number of collinear triples in $A$:
\begin{equation}\label{e3sph}
\E_3\ll |A|^2k_0 + |A|k_0^2\,.\end{equation}

The rest of the proof deals with rectangles of the first type. It follows the same strategy of counting right angles from each point $z\in A$ to pairs of other points of $A$ as in the proof of Theorem \ref{f:inc_d-1}. Each fixed $z\in A$ gives rise to a point-plane incidence problem in the hyperplane at infinity, set up in exactly the same way as the proof of Theorem  \ref{s_corners} sets up a point-line incidence count in the plane at infinity. 

More precisely, for a fixed $z$ define the set $Q_z$ of ideal points, corresponding to non-isotropic lines $zx$. The number of first type rectangles with a vertex at $z$ is bounded in terms of the number of incidences of $Q_z$ with a multiset of planes of total weight $\sim|A|$, the number of distinct planes, owing to Lemma \ref{easy} (unless a large proportion of $A$ lies on a single isotropic line) being $\gg|Q_z|$.

By Theorem \ref{Misha+} the number of incidences is for each $z$ is \begin{equation}\label{intsph}
O\left( \frac{|A|^2}{p} + |A|^{\frac{3}{2}} + k|A| \right) \,,
\end{equation}
where $k$ is the maximum number of collinear points,  that is the maximum number of points of $A$ in the intersection of  $\mathcal{S}^3_t$ with a non-isotropic plane through $z$. If the plane is semi-isotropic, meeting $\mathcal S^3_t$ in two parallel isotropic lines, we have a bound $k\leq k_0$.

If the plane is ordinary, then its intersection with $\mathcal{S}^3_t$ is either two isotropic lines meeting at $z$, in which case they cannot be mutually orthogonal, which is irrelevant for the count of first type rectangles. Otherwise it is a conic via $z$. Let us show that in this case, as in the proof of  Theorem \ref{l:par_energy}, we can effectively set $k=\sqrt{|A|}.$

Following the proof of Theorem \ref{l:par_energy}, let $L'$ be the set of plane conic curves on $\mathcal S^3_t$, supporting $\geq 10\sqrt{|A|}$ points of $A$ each, and $A'$ the part of $A$ supported on the union of conics in $L'$. By inclusion-exclusion principle $|L'|\ll\sqrt{|A|}.$ Now, in the application of Theorem \ref{Misha+} with each fixed $z$, we forbid the set of lines arising at infinity after projecting the conics in $L'$, containing $z$. For this restricted number of incidences one has $k=k_0$ in estimate \eqref{intsph}.

This means that the quantity
\begin{equation}\label{e1sph}
\E_1 := O\left( \frac{|A|^3}{p} + |A|^{\frac{5}{2}} +|A|^2k_0\right)
\end{equation}
bounds the number of first type rectangles, such that one of the lines, containing a non-null side of the rectangle supports $O(\sqrt{|A|})$ points of $A$.

To complement the latter bound, let $\E_2$ be the total number of the first type rectangles that remain.
Similar to the proof of Theorem \ref{l:par_energy}, one has $\E_2\ll \E(A,A')$. Let $A_c$ be the subset of $A$ supported on one such conic. If we show that   
$$
\E(A,A_c)\ll |A||A_c|,
$$
then we are done by repeating the Cauchy-Schwarz estimate \eqref{f:E_par2}.

But this follows immediately, since $L'$ is the set of plane conics, rather than isotropic lines. Suppose $x$ is a vertex of a rectangle on the conic, $y\in A$, $xy$ being the diagonal of the rectangle, and one is looking for $z$ on the conic, so that
$$
(x-z)\cdot(z-y)=0,
$$
Let $\pi$ be the plane containing the conic. If there were three distinct solutions $z_1,z_2,z_3$ of the latter equation, they would be collinear in $\pi$, but a line may meet a conic at $\leq 2$ points.

Thus $\E_2\ll |A|^{\frac{5}{2}}$: this plus \eqref{e1sph} and \eqref{e3sph} completes the proof of 
 of Theorem \ref{l:sph_energy}.
\end{proof}

\end{document}